\documentclass[12pt]{article}
\usepackage{a4}
\usepackage{amsmath}
\usepackage{amssymb}
\usepackage{amsfonts}
\usepackage{amsthm}
\usepackage{graphicx}
  \newtheorem{thm}{Theorem}
 \newtheorem{cor}[thm]{Corollary}
 \newtheorem{prop}[thm]{Proposition}

 \newtheorem{rem}[thm]{Remark}
 
\DeclareMathOperator{\Log}{Log}
 \DeclareMathOperator{\Arg}{Arg}
 
\newcommand{\C}{\mathbb{C}}
\newcommand{\R}{\mathbb{R}}

\def\C{\mathbb{C}}
\def\R{\mathbb{R}}

\def\H{\mathbb{H}}
\def\la{\lambda}

\author {Christian Berg}
\title {A complete Bernstein function related to the fractal dimension of Pascal's pyramid modulo a prime}

\date{27.1.2025}

\begin{document}
\maketitle
\begin{center}
Dedicated to the memory of Bent Fuglede (1925-2023)
\end{center}
\begin{abstract} Let $f_r(x)=\log(1+rx)/\log(1+x)$ for $x>0$. We prove that $f_r$ is a complete Bernstein function for $0\le r\le 1$ and a Stieltjes function for $1\le r$.
This answers a conjecture of David Bradley that $f_r$ is a Bernstein function when $0\le r\le 1$.
\end{abstract}



\section{Introduction and main results}
For $0<r$ we consider the function 
\begin{equation}\label{eq:fr}
f_r(x):=\frac{\log(1+rx)}{\log(1+x)},\quad x>0.
\end{equation}
It extends by continuity to the interval $(\max(-1,-1/r), 0]$  with the value $f_r(0)=r$.
Furthermore, $\lim_{x\to\infty}f_r(x)=1$.

In a study of the fractal dimension of  certain sets in number theory coming from the number of multinomial coefficients congruent to a given residue modulo a prime,  David Bradley is led to the conjecture, that $f_r$ is a Bernstein function when $0<r<1$, see \cite{Brad}.  We prove that this is true as part of the following result:

\begin{thm}\label{thm:main0} The function $f_r$ is a complete Bernstein function
when $0<r<1$  and a Stieltjes function when $r>1$.
\end{thm} 

Bernstein functions are by definition $C^\infty$-functions $f:(0,\infty)\to[0,\infty)$ satisfying
\begin{equation}\label{eq:BF}
(-1)^{n-1}f^{(n)}(x)\ge 0,\quad x>0, n=1,2,\ldots.
\end{equation} 
In particular, Bernstein functions are non-negative, increasing and concave. To verify the inequalities \eqref{eq:BF} for $f_r$ seems a complicated task, because one has to differentiate a quotient. We prove instead the stronger statement, that $f_r$ belongs to the subclass $\mathcal{CBF}$ of complete Bernstein functions studied in \cite{S:S:V}. Our proof uses complex analysis. 

Bernstein functions have a long history, and they play an important role in probability theory and potential theory, see \cite{S:S:V} and \cite{B:F}. 

We also need the class ${\mathcal S}_{\la}$ of generalized Stieltjes functions of order $\la>0$ defined as the functions $f:(0,\infty)\to [0,\infty)$ having a representation of the form
\begin{equation}\label{eq:Sl}
f(x)=c+\int_0^\infty\frac{d\mu(t)}{(x+t)^\la},
\end{equation}
where $c\ge 0$, and $\mu$  is a positive measure on $[0,\infty)$ such that the integral in \eqref{eq:Sl} converges for $x>0$. When $\la=1$  the functions in $\mathcal S=\mathcal S_1$ are just called Stieltjes functions. A function in the class $\mathcal S_\la$ clearly extends to a holomorphic function in the cut plane $\C\setminus (-\infty,0]$. For more information on these functions see \cite{K:P2}. We mention in particular that $\mathcal S\subset \mathcal S_2$.

Complete Bernstein functions are characterized by  six equivalent conditions given in Theorem 6.2 of \cite{S:S:V}. We shall establish condition (iv), stating that they are Pick functions with certain properties. For the convenience of the reader we have included some key results about Pick functions in Section 2.

We denote by $\Log$ the principal branch of the  logarithm, which is  holomorphic in the cut plane $\C\setminus (-\infty,0]$ and defined by 
$$
\Log(z)=\log|z| + i\Arg(z), 
$$
where $\Arg$ is the principal argument taking its values in $(-\pi,\pi)$.
The expression
$$
f_r(z):=\frac{\Log(1+rz)}{\Log(1+z)}
$$
is a holomorphic extension of $f_r$ to the cut plane $\C\setminus (-\infty, \max(-1,-1/r)]$, ($z=0$ being a removable singularity).  

We first give more details about the case $0<r<1$. 

\begin{thm}\label{thm:main1} For each $0<r<1$ the function $f_r$ is a complete Bernstein function given for $z\in\C\setminus (-\infty,-1]$ by
\begin{equation}\label{eq:fr2}
f_r(z)=\frac{\Log(1+rz)}{\Log(1+z)}=r+\int_1^\infty \frac{z}{z+t}\sigma_r(t)\,dt,
\end{equation}
where 
\begin{equation*}
 \sigma_r(t)=\left\{\begin{array}{cl}
 \frac{-\log(1-rt)}{t\left(\log^2 (t-1)+\pi^2\right)}, & 1<t<1/r,\\
 \infty, & t=1/r,\\
\frac{\log((t-1)/(rt-1))}{t\left(\log^2 (t-1)+\pi^2\right)}, & 1/r < t. 
    \end{array}
  \right.
  \end{equation*}
 On the interval $(-1,\infty)$ the function $f_r$ is strictly concave and strictly increasing from $0$ to $1$.   
  The function $1-f_r(z)$ is a Stieltjes function given by
  \begin{equation}\label{eq:fr3}
  1-f_r(z)=\int_1^\infty \frac{t\sigma_r(t)}{z+t}\,dt,
  \end{equation}
  and $f_r(z)/z$ is a Stieltjes function given by
 \begin{equation}\label{eq:fr31}
  f_r(z)/z=\int_0^\infty \frac{d\mu_r(t)}{z+t}, \quad \mu_r=r\delta_0+1_{(1,\infty)}(t)\sigma_r(t)\,dt,
  \end{equation}
  where $\delta_0$ is the Dirac measure at $0$. 
The function $\sigma_r$ tends to 0 for $t\to 1^+$ and for $t\to\infty$. The following formulas hold
\begin{equation}\label{eq:sig}
\int_1^\infty \sigma_r(t)\,dt=1-r,\quad \int_1^\infty\frac{\sigma_r(t)\,dt}{1+t^2}=
 \frac{\log 2 \log(1+r^2)+\pi\arctan(r)}{\log^2(2)+\pi^2/4}-r.
\end{equation}
\end{thm} 

We include a plot of $\sigma_r$, where $r=1/2$. Note that we have different scaling on  
the two subintervals $[1, 2),  (2, 5]$.

\medskip
\begin{center}
\includegraphics[scale=0.3]{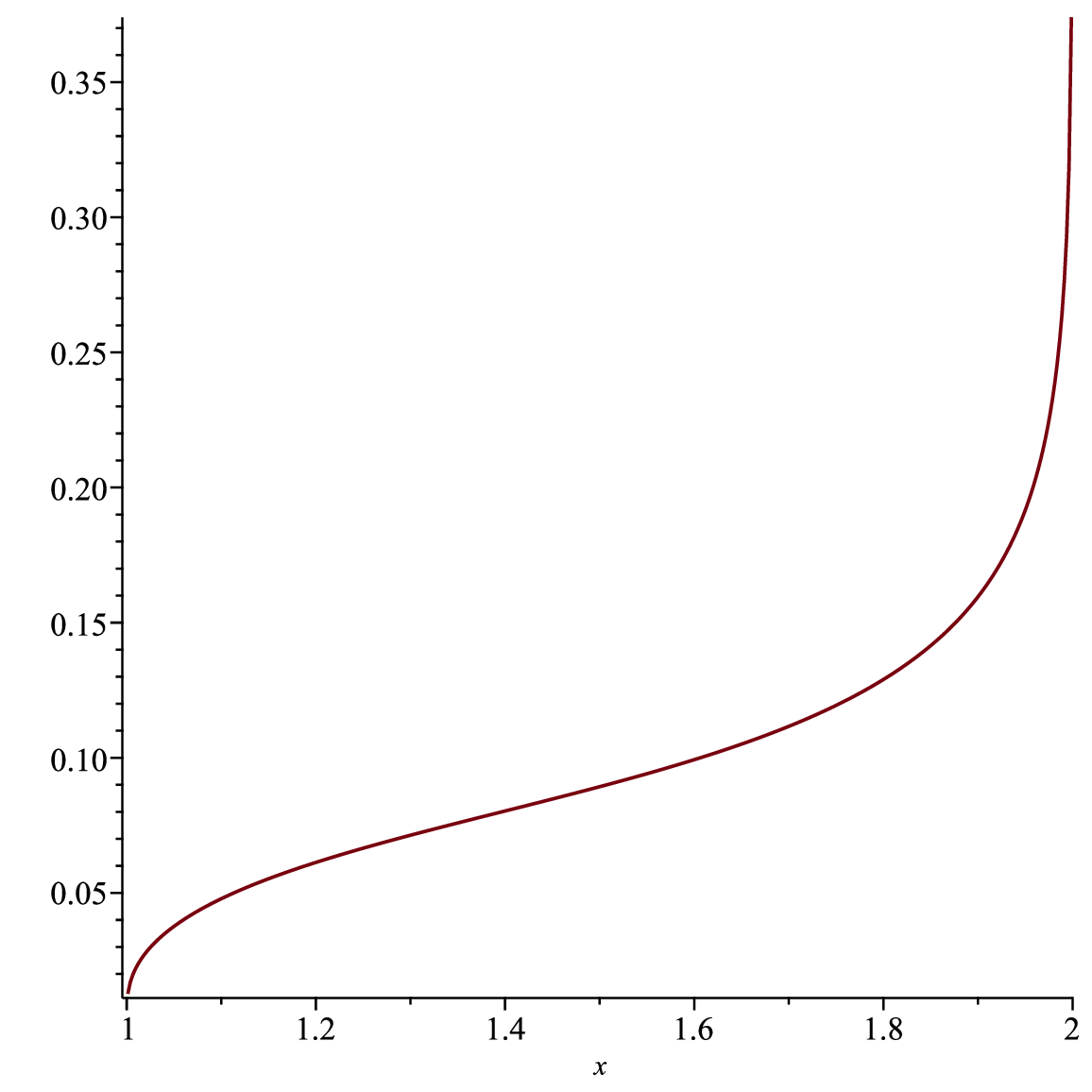}
\includegraphics[scale=0.3]{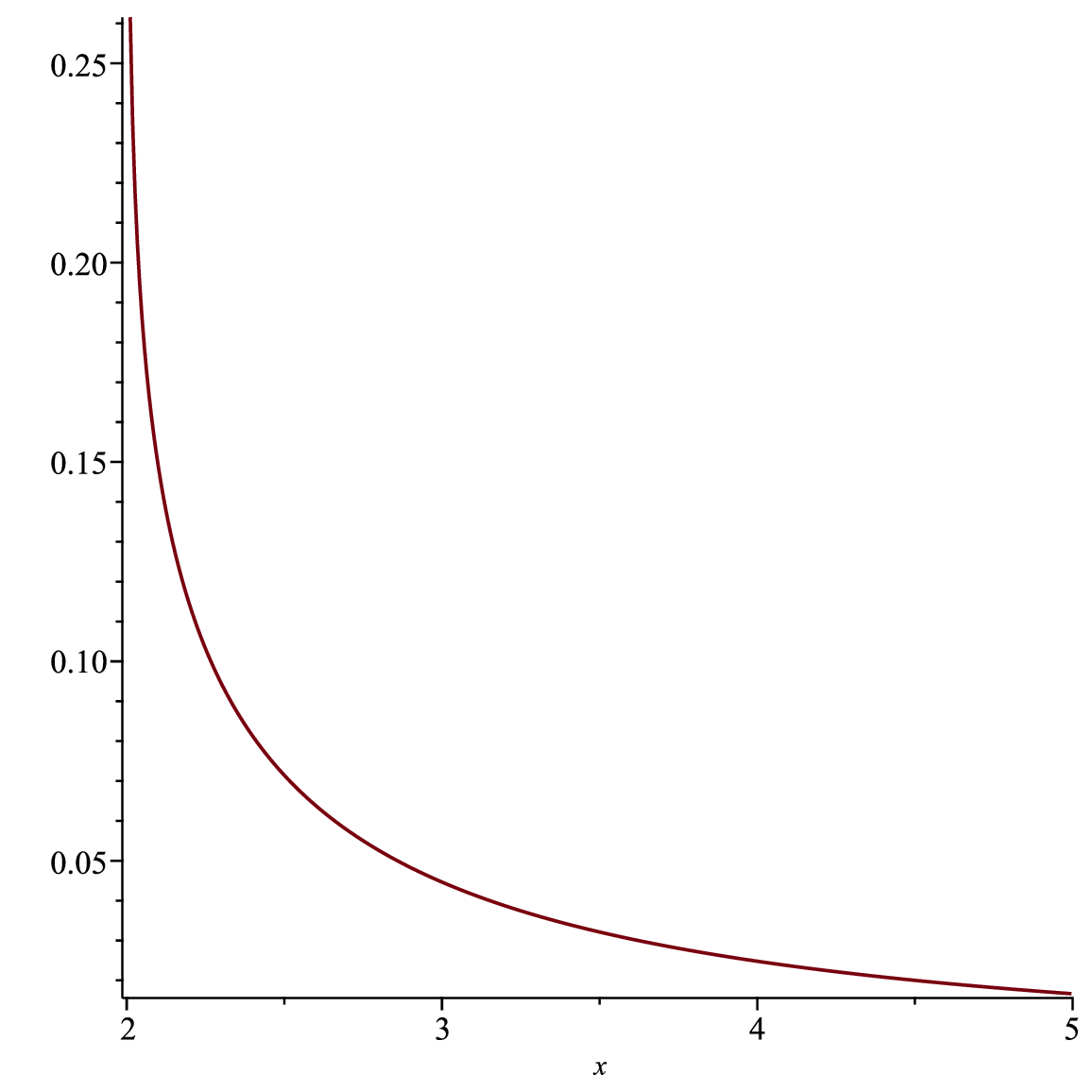}
\end{center}
\begin{center}
Plot of $\sigma_{0.5}$
\end{center}
\begin{rem}\label{thm:Rem1}{\rm There exists a number $r_0\approx 0.1$ such that
$\sigma_r(t)$ is increasing on $(1,1/r)$ when $r_0\le r<1$, but when $0<r<r_0$ then $\sigma_r(t)$ is decreasing in a certain subinterval $[a_r,b_r]\subset (1,1/r)$ and otherwise increasing. 
}
\end{rem}

We next give more details about the case $r>1$. 

\begin{thm}\label{thm:main2} For $r>1$ the function $f_r$ is a Stieltjes function given for $z\in\C\setminus(-\infty,-1/r]$ by
\begin{equation}\label{eq:41}
f_r(z)=1+\int_{1/r}^\infty\frac{\omega_r(t)}{z+t}\,dt,
\end{equation}
where
\begin{equation*}
 \omega_r(t)=\left\{\begin{array}{cl}
\frac{-1}{\log(1-t)}, & 1/r \le t <1,\\
0, & t=1,\\
\frac{\log((rt-1)/(t-1))}{\log^2(t-1)+\pi^2}, & 1<t.
 \end{array}
  \right.
\end{equation*}
On the interval $(-1/r,\infty)$ the function $f_r$ is strictly convex and strictly  decreasing from $\infty$ to $1$ and 
\begin{equation}\label{eq:42}
\int_{1/r}^\infty \frac{\omega_r(t)}{t}\,dt=r-1.
\end{equation}
\end{thm}

We include a plot of $\omega_r$, where $r=2.5$. Note that we have different scaling on  
the two subintervals $[0.4, 1),  (1, 1.5]$.
\medskip

\begin{center}
\includegraphics[scale=0.3]{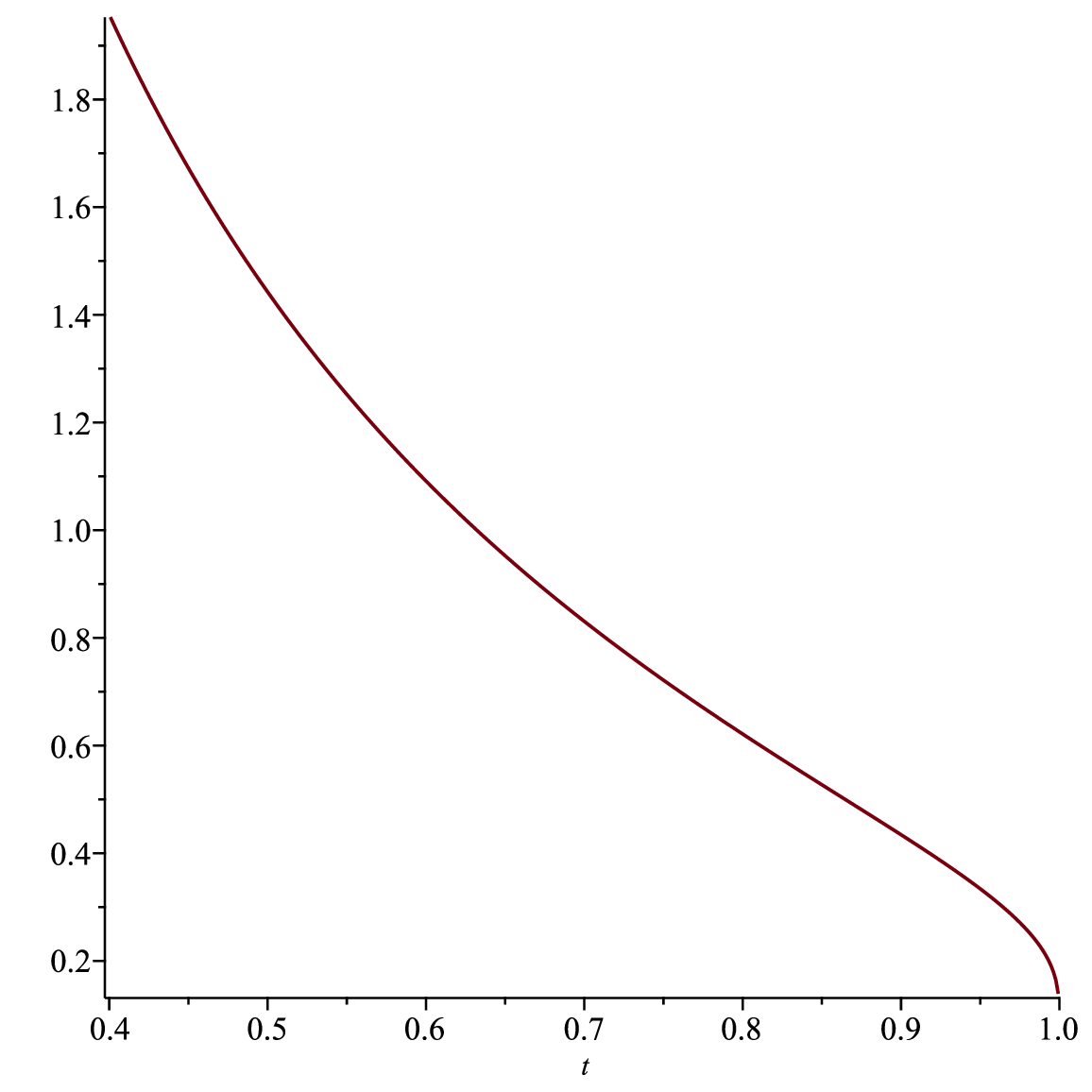}  \includegraphics[scale=0.3]{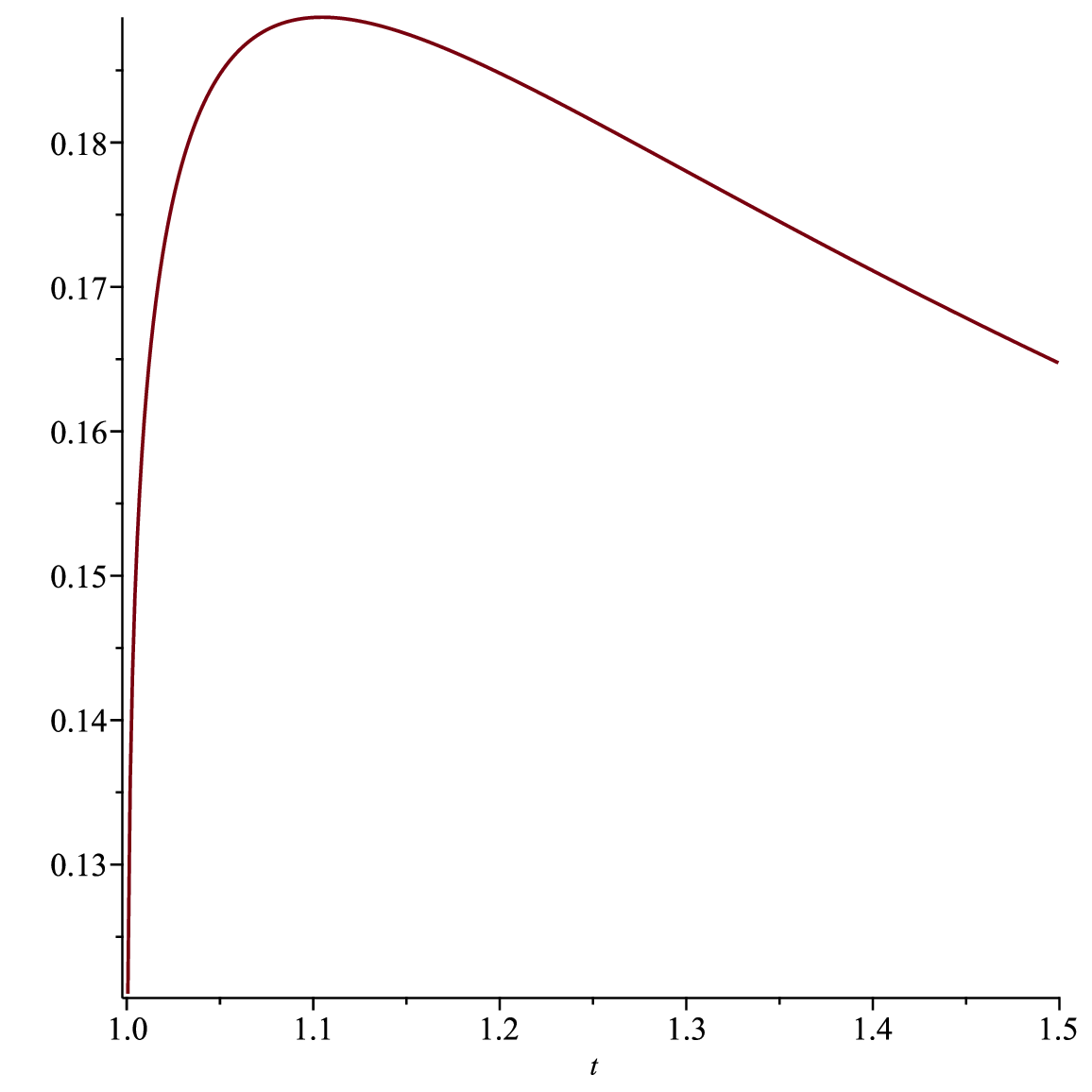}
\end{center}
\begin{center}
Plot of $\omega_{2.5}$
\end{center}

\begin{rem}\label{thm:Rem2} {\rm 
If $f(z)$ is a complete Bernstein function or a Stieltjes function, then so is $f(cz)$ for any $c>0$. In particular, for $0<a<b$
$$
f_{a/b}(bz)=\frac{\Log(1+az)}{\Log(1+bz)}
$$ 
is a complete Bernstein function, and
$$
f_{b/a}(az)=\frac{\Log(1+bz)}{\Log(1+az)}
$$
is a Stieltjes function.
}
\end{rem}

Proofs of the theorems are given in Section 3.

\section{Preliminaries about Pick functions}
Let $\H:=\{z\in\C \mid \Im z>0\}$ denote the open upper half-plane.

A holomorphic function $f:\H\to\C$ is called a Pick function if its imaginary part $\Im f(z)\ge 0$ for all $z\in\H$.

For a Pick function $f$, $\Im f$ is a non-negative  harmonic function, and it has the mean value property. It is therefore either identically zero  or strictly positive on $\H$. This means that a  Pick function is either a real constant function or it maps $\H$ into itself. Pick functions are treated in \cite{Do}, \cite{S:S:V} and in the Appendix to \cite{Sch}.    
 
 A Pick function $f$ has the representation
\begin{equation}\label{eq:Pi}
f(z)=\alpha+\beta z +\int_{-\infty}^\infty \left(\frac{1}{t-z}-\frac{t}{1+t^2}\right)\,d\mu(t), \quad z\in\H,
\end{equation}
where $\alpha\in\R,\beta\ge 0$ and $\mu$ is a positive measure on $\R$ satisfying the integrability condition $\int (1+t^2)^{-1}\,d\mu(t)<\infty$. The triple $(\alpha,\beta,\mu)$ is uniquely determined by $f$ as
\begin{equation}\label{eq:Pi2}
\alpha=\Re f(i),\quad \beta=\lim_{y\to\infty} f(iy)/iy,\quad 
\lim_{y\to 0^+} (1/\pi)\Im f(x+iy)\,dx=\mu,
\end{equation}
where the last convergence is in the vague topology for measures meaning that
\begin{equation}\label{eq:Pi21}
\lim_{y\to 0^+}\frac{1}{\pi}\int \Im f(x+iy)g(x)\,dx=\int g\,d\mu
\end{equation}
for all continuous functions $g:\R\to\C$ with compact support. To see this we note that using convolution we have
$$
\frac{1}{\pi}\int\Im f(x+iy)g(x)\,dx=\frac{\beta y}{\pi}\int g(x)\,dx +\int P_y\star g(t)\,d\mu(t),
$$
where
$$
 P_y(x)=\frac{1}{\pi}\frac{y}{x^2+y^2},\quad y>0, x\in\R
$$
is the Poisson kernel for the upper half-plane.

Equation \eqref{eq:Pi21} now follows because $\lim_{y\to 0^+}P_y\star g(t)=g(t)$  for  every $t\in\R$ (and even uniformly in $t$). Furthermore, we can apply Lebesgue's dominated convergence theorem. To see the latter, if  $g$ has 
  compact support contained in the interval $[a,b]$ and supremum norm $C:=||g||_\infty$,
we find  for $0<y<1, t\in\R$,
$$
|P_y\star g(t)|\le \frac{1}{\pi}\int_a^b\frac{y|g(x)|}{(t-x)^2+y^2}\,dx\le
\left\{\begin{array}{cl}
C, &  t \in\R,\\
\frac{C(b-a)}{\pi (a-t)^2}, & t<a,\\
\frac{C(b-a)}{\pi (t-b)^2}, & t>b,
\end{array}
\right.
$$
so $|P_y\star g(t)|$ has a $\mu$-integrable majorant because of the assumption $\int (1+t^2)^{-1}\,d\mu(t)<\infty$.

This implies in particular, that if there is an open interval $I\subseteq\R$ such that
\begin{equation}\label{eq:Pi3}
\lim_{y\to 0^+}(1/\pi)\Im f(t+iy)=\varphi(t),\quad t\in I
\end{equation}
and the convergence is uniform on compact subsets of $I$, then the restriction of $\mu$ to the interval $I$ has the continuous density $\varphi$ with respect to Lebesgue measure.

The measure $\mu$ has mass $m\ge 0$  at a point $a\in\R$ if and only if
\begin{equation}\label{eq:Pi4}
\lim_{y\to 0^+}y\Im f(a+iy)=m.
\end{equation} 
In fact,
$$
y\Im f(a+iy)=\beta y^2+\int \frac{y^2}{(t-a)^2+y^2}\,d\mu(t)\to \mu(\{a\})
$$
for $y\to 0^+$ because 
$$
\frac{y^2}{(t-a)^2+y^2}\le \frac{2}{1+(t-a)^2},\quad 0<y<1, t\in\R,
$$
and therefore $2(1+(t-a)^2)^{-1}$ is an integrable majorant  with respect to $\mu$. 
 
 For a Stieltjes function $f$ given by \eqref{eq:Sl} with $\la=1$, we get
 $$
 -(1/\pi)\Im f(-x+iy)=P_y\star \mu(x),
 $$ 
and we get similar results for  Stieltjes functions as for Pick functions.

\section{Proofs}

We need the following boundary minimum principle for harmonic functions, see e.g. \cite[p.5]{Doob}.

\begin{prop}\label{thm:min} Let $u$ be a harmonic function in $\H$ satisfying

(i) $\liminf_{z\to t, z\in\H}u(z)\ge 0$ for each $t\in\R$, 

(ii) $\liminf _{|z|\to\infty, z\in\H} u(z)\ge 0$.

\noindent Then $u(z)>0$ for all $z\in\H$ unless $u$ is identically zero. 
\end{prop}

{\it Proof of Theorem~\ref{thm:main1}}. 

From the expression
$$
f_r(z)=\frac{\log|1+rz|+i\Arg(1+rz)}{\log|1+z|+i\Arg(1+z)},\quad  z\in\H,
$$
we see that $f_r(z)\to 1$  and hence $\Im f_r(z)\to 0$ for $|z|\to \infty$ in $\H$.
 
 The harmonic function $\Im f_r(z)$ in $\H$  is given by
\begin{equation}\label{eq:Ifr}
\Im f_r(z)=\frac{\log|1+z|\Arg(1+rz)-\log|1+rz|\Arg(1+z)}{\log^2|1+z|+\Arg^2(1+z)}.
\end{equation}
It does not seem possible to see directly that this expression is positive, but it will follow from Proposition~\ref{thm:min}.
We claim that for $t\in\R$
$$
\varphi_r(t):=\lim_{z\to t, z\in\H} (1/\pi)\Im f_r(z)=\left\{\begin{array}{cl}
0, & t\ge -1,\\
\frac{-\log(1+rt)}{\log^2(-t-1)+\pi^2},    & -1/r<t<-1,\\
   \infty, & t=-1/r,\\
   \frac{\log((-t-1)/(-1-rt))}{\log^2(-t-1)+\pi^2}, & t<-1/r.
     \end{array}
  \right.
$$
For this we use that $\Arg(z)\in (0,\pi)$ for $z\in \H$ and  
$$
\lim_{z\to t, z\in\H}\Arg(z)=\left\{\begin{array}{cl}
0, & t>0,\\
\pi, & t<0,
\end{array}
\right.  \quad \lim_{z\to 0, z\in\H} \log|z|=-\infty.
$$
The convergence to $\varphi_r(t)$ is uniform for $t$ in compact subsets of the open intervals $(-\infty,-1/r), (-1/r,-1), (-1,\infty)$, and the expressions of $\varphi_r(t)$ to the right are non-negative for all $t\in\R$.

We have shown that the conditions of Proposition~\ref{thm:min} are satisfied for $\Im f_r$, and the $\liminf$ condition is a  true limit in all cases. Therefore $\Im f_r(z)>0$ for all $z\in\H$.
This shows that $f_r$ is a Pick function with a representation
\begin{equation}\label{eq:fr4}
f_r(z)=\Re f_r(i)+\int_{-\infty}^\infty \left(\frac{1}{t-z}-\frac{t}{1+t^2}\right)\,d\mu_r(t),
\end{equation} 
see Section 2, because
$$
\beta_r:=\lim_{y\to\infty}\frac{f_r(iy)}{iy}=0,
$$
as is easily seen. We have
$$
\alpha_r:=\Re f_r(i)=\frac{\log 2 \log(1+r^2)+\pi\arctan(r)}{\log^2(2)+\pi^2/4}>0.
$$
The measure $\mu_r$ has the density $\varphi_r$ in the open intervals $(-\infty, -1/r), (-1/r, -1)$ and $(-1,\infty)$ because of \eqref{eq:Pi3}, and there are no masses at the points $-1/r,-1$ by \eqref{eq:Pi4}. Replacing $t$ by $-t$  we finally get

\begin{equation}\label{eq:fr5}
f_r(z)=\alpha_r +\int_1^\infty\left(\frac{t}{1+t^2}- \frac1{t+z}\right)\varphi_r(-t)\,dt.
\end{equation}
In the notation of formula (6.8) in \cite{S:S:V} we have $\varphi_r(-t)\,dt=(1+t^2)d\rho(t)$, and \eqref{eq:fr2} follows with $\sigma_r(t)=\varphi_r(-t)/t$. Putting $z=0$ in \eqref{eq:fr5} we get
$$
r=\alpha_r-\int_1^\infty\frac{\varphi_r(-t)}{t(1+t^2)}\,dt=\alpha_r-\int_1^\infty\frac{\sigma_r(t)}{1+t^2}\,dt,
$$
and the second formula in \eqref{eq:sig} follows.

The formula in Theorem~\ref{thm:main1} for $\sigma_r$ shows that it is integrable over the interval $(1,\infty)$. From  formula \eqref{eq:fr2} we then get
$$
f_r(z)=r + \int_1^\infty \sigma_r(t)\,dt - \int_1^\infty\frac{t\sigma_r(t)}{z+t}\,dt, 
$$
so for $z\to\infty$ on the real axis we get
$$
1=r+\int_1^\infty\sigma_r(t)\,dt,
$$
and \eqref{eq:fr3} follows together with the first formula in \eqref{eq:sig}.

From \eqref{eq:fr3} it is clear that $f_r$ is strictly increasing and concave on $(-1,\infty)$.
Formula \eqref{eq:fr31} follows immediately from \eqref{eq:fr2}.
$\square$

\begin{cor} For $0<r<1$ we have
\begin{equation}\label{eq:cor1}
0<f_r(x)<r\frac{1+x}{1+rx},\quad x>-1,
\end{equation}
and
\begin{equation}\label{eq:cor2}
\log|1+z|\Arg(1+rz) > \log|1+rz|\Arg(1+z),\quad z\in \H.
\end{equation}
\end{cor}

\begin{proof} The derivative of $f_r$ satisfies
\begin{equation}\label{eq:cor3}
f_r'(x)\log^2(1+x)=\log(1+x)\frac{r}{1+rx}-\log(1+rx)\frac{1}{1+x},
\end{equation}
and since $f_r$ is strictly increasing, we get \eqref{eq:cor1}.

Formula \eqref{eq:cor2} follows from \eqref{eq:Ifr} because $\Im f_r(z)>0$ for $z\in\H$.
\end{proof}

\begin{rem}\label{thm:Rem3} {\rm There is an important subclass of $\mathcal{CBF}$ consisting of the
Thorin-Bernstein functions, denoted $\mathcal{TBF}$ and studied in Chapter 8 of \cite{S:S:V}. According to the proof of Theorem 8.2 in \cite{S:S:V} we can conclude that $f_r\notin \mathcal{TBF}$ for $0<r<1$ because a necessary and sufficient condition for $f_r$ to be a Thorin-Bernstein functions is that $t\sigma_r(t)$ is an increasing function for $t>0$, and this is never true, since the function is infinity for $t=1/r$ and finite for larger $t$ and tends to 0 for $t\to\infty$. 

A tedious calculation shows that $t\sigma_r(t)$ is increasing on $(1,1/r)$. Compare this with Remark~\ref{thm:Rem1}.

We also recall that by Theorem 8.2 in \cite{S:S:V} $f\in\mathcal{CBF}$ belongs to $\mathcal{TBF}$ if and only if $f'\in\mathcal S$.

 From \eqref{eq:fr2} or \eqref{eq:fr3} we get
$$
f_r'(z)=\int_1^\infty \frac{t\sigma_r(t)}{(z+t)^2}\,dt,
$$
showing that $f_r'$ is a generalized Stieltjes function of order 2, which is not surprising as this holds for all complete Bernstein functions. In conclusion we have $f_r'\in\mathcal S_2\setminus\mathcal S$ for $0<r<1$.
 
We recall that  generalized Stieltjes functions of order 2 are logarithmically completely monotonic by a deep theorem of Steutel and Kristiansen, see Theorem 2.1 in \cite{B:K:P} for details. Consequently, $f_r$ belongs to the class of Horn-Bernstein functions, introduced in \cite{B:P}.

The Bernstein representation for $f_r$, cf. Equation (6.1) in \cite{S:S:V}, is given as
$$
f_r(x)=r+\int_0^\infty (1-e^{-xs})m_r(s)\,ds,\quad x>0,
$$
where $m_r$ is the completely monotonic function
$$
m_r(s)=\int_1^\infty e^{-ts}t\sigma_r(t)\,dt,\quad s>0
$$
according to Theorem 6.2 in \cite{S:S:V}.
}
\end{rem}

{\it Proof of Theorem~\ref{thm:main2}}.  We first recall from Chapter 7 in \cite{S:S:V} that for $f\in \mathcal{CBF}$ which is not identically zero, then $z/f(z)$ also belongs to $\mathcal{CBF}$  and is called  the conjugate of $f$. Furthermore, by Theorem 6.2 in \cite{S:S:V} we then get that $1/f\in\mathcal S$.
 
 Assume now $r>1$. From \eqref{eq:Son} we have
$$
f_r(x)=\frac1{f_{1/r}(rx)},\quad x>0,
$$
so $f_r$ is the reciprocal of a complete Bernstein function and therefore $f_r\in\mathcal S$. Since $f_r(x)\to 1$ for $x\to\infty$ we have
$$
f_r(z)=1+\int_0^\infty \frac{d\nu_r(t)}{z+t}, \quad z\in\H
$$
for a positive measure $\nu_r$ on $[0,\infty)$.  
From Section 2 we know that the positive measures 
$$
-\frac{1}{\pi}\Im f_r(-t+iy)\,dt
$$
converge vaguely to $\nu_r$ for $y\to 0^+$.
From \eqref{eq:Ifr} we get for $y>0$
\begin{eqnarray*}
\lefteqn{-\frac{1}{\pi}\Im f_r(-t+iy)=}\\
&&\frac{1}{\pi}\frac{\log|1-rt+iry|\Arg(1-t+iy)-\log|1-t+iy|\Arg(1-rt+iry)}
{\log^2|1-t+iy|+\Arg^2(1-t+iy)},
\end{eqnarray*} 
which   for $y\to 0^+$ converges to
\begin{equation*}
 \omega_r(t)=\left\{\begin{array}{cl}
 0, & t<1/r,\\
\frac{-1}{\log(1-t)}, & 1/r < t <1,\\
0, & t=1,\\
\frac{\log((rt-1)/(t-1))}{\log^2(t-1)+\pi^2}, & 1<t,
 \end{array}
  \right.
\end{equation*}
and the convergence is uniform for $t$ in compact subsets of the open intervals 
$(-\infty,1/r), (1/r,1), (1,\infty)$.
Using the results of Section 2, it follows  that $\nu_r$ has the density $\omega_r(t)$ with respect to Lebesgue measure, and there are no masses at $t=1/r,1$, hence \eqref{eq:41} follows.

From \eqref{eq:41} we see that $f_r$ is strictly decreasing and convex on $(-1/r,\infty)$, and for $z=0$ we get
\eqref{eq:42}.
$\square$

\begin{cor} For $r>1$ we have
\begin{equation}\label{eq:2cor1}
f_r(x)>r\frac{1+x}{1+rx},\quad x>-1/r,
\end{equation}
and
\begin{equation}\label{eq:2cor2}
\log|1+z|\Arg(1+rz) < \log|1+rz|\Arg(1+z),\quad z\in \H.
\end{equation}
\end{cor}

\begin{proof} Equation \eqref{eq:2cor1} follows from \eqref{eq:cor3} using that $f_r$ is strictly decreasing.

Equation \eqref{eq:2cor2} follows from \eqref{eq:Ifr} because $\Im f_r(z)<0$ for $z\in\H$, since $f_r$ is a Stieltjes function.
\end{proof}

\section{A convolution equation}

The Laplace transform of a positive measure $\mu$ on $[0,\infty)$ is the function
$$
\mathcal L\mu(x)=\int_0^\infty e^{-xt}\,d\mu(t),\quad x>0,
$$
provided that the integral is finite for all $x>0$.
If $\mu$ has a continuous density $f:(0,\infty)\to [0,\infty)$ with respect  to 
Lebesgue measure $m$ on $(0,\infty)$, i.e., $d\mu=f dm$, we write $\mathcal L f= 
\mathcal L \mu$.

With this in mind we have 
$$
\mathcal L(\mathcal L\mu)(y)=\int_0^\infty e^{-yx}\mathcal L\mu(x)\,dx=\int_0^\infty\frac{d\mu(t)}{y+t},\quad y>0.
$$

Note that the family of functions $f_r, r>0$ satisfies
\begin{equation}\label{eq:Son}
f_r(x) f_{1/r}(rx)=1,\quad r,x>0.
\end{equation}
Assume now that $0<r<1$ and rewrite \eqref{eq:Son} as
\begin{equation}\label{eq:Son1}
\frac{f_r(x)}{x} f_{1/r}(rx)=\frac{1}{x}=\mathcal L m(x)=\mathcal L 1(x),\quad x>0.
\end{equation}
 Using the Laplace transformation, Equation \eqref{eq:fr31} can be written
\begin{equation}\label{eq:Lap1}
\frac{f_r(x)}{x}=\mathcal L(\mathcal L\mu_r)(x)=\mathcal L g_r(x),\quad g_r(s)=r+\mathcal L\sigma_r(s),
\end{equation}
and Equation \eqref{eq:41} can be written
\begin{equation}\label{eq:Lap2}
f_{1/r}(rx)=1+\int_1^\infty\frac{\omega_{1/r}(rt)}{x+t}\,dt=1+\mathcal L h_r(x),
\end{equation}
where
$$
 h_r(s)=\mathcal L \omega_{1/r}(rt)(s)=\int_1^\infty e^{-st}\omega_{1/r}(rt)\,dt,\quad s>0.
$$
Using well-known properties of Laplace transformation,
 we see that \eqref{eq:Son1} is equivalent to the convolution equation

$$
g_r\star(\delta_0+h_r)(x)=1,\quad x>0
$$
or
\begin{equation*}
g_r(x)+\int_0^x g_r(s)h_r(x-s)\,ds=1,\quad x>0.
\end{equation*}

\section{A question answered in the negative}

Since $f(x)=\log(1+x)$ is a complete Bernstein function, it is a natural question if the following extension of Theorem 1 holds:
$$
f\in\mathcal{CBF}\implies f(rx)/f(x)\in\mathcal {CBF}\;\mbox{when}\, 0<r<1.
$$
We shall see that the answer to this question is negative.

First notice that if such an implication holds, then necessarily $f(0)=0$, for otherwise letting $r\to 0$ we would obtain $1/f(x)\in\mathcal {CBF}$, and under the assumption $f\in\mathcal{CBF}$ this is only possible, when $f$ is constant.

A counterexample to the implication is obtained, when
$f(x)=x/(x+1)+x/(x+3)$ which is a complete Bernstein function. A small calculation shows that
$$
\frac{f(x/4)}{f(x)}=\frac{(x+1)(x+3)(x+8)}{(x+2)(x+4)(x+12)}= 1-\frac{3/10}{x+2} - \frac{3/4}{x+4}- \frac{99/20}{x+12}.
$$
Since we have negative coefficients in the partial fraction decomposition, the conclusion is that the function $f(x/4)/f(x)\notin\mathcal{CBF}$.

\bigskip
{\bf Acknowledgement} The author wants to thank David Bradley for having suggested the title of the present work, which is  inspired by \cite{Brad}. The author also wants to thank Ren{\'e} Schilling for useful comments to the manuscript.

\noindent
Christian Berg\\
Department of Mathematical Sciences, University of Copenhagen\\
Universitetsparken 5, DK-2100 Copenhagen, Denmark\\
email: {\tt{berg@math.ku.dk}}

\end{document}